\documentclass[a4paper,11pt]{article}

\pdfoutput=1

\usepackage[a4paper]{geometry}
\usepackage{fullpage,mathtools,amssymb,amsthm,xcolor}
\usepackage[style=alphabetic,backref]{biblatex}
\usepackage[colorlinks=true]{hyperref}
\usepackage[capitalize,nameinlink,noabbrev]{cleveref}

\crefname{subsection}{Subsection}{subsections}

\theoremstyle{plain}
	\newtheorem{thrm}{Theorem}[section]
	\newtheorem{lem}[thrm]{Lemma}
	\newtheorem{cor}[thrm]{Corollary}
	\newtheorem{prop}[thrm]{Proposition}

\theoremstyle{definition}

	\newtheorem{rmrk}[thrm]{Remark}
	
	\newtheorem{ex}[thrm]{Example}

\newcommand{\R}{\mathbb{R}}

\newcommand{\N}{\mathbb{N}}
\newcommand{\Z}{\mathbb{Z}}

\newcommand{\Tcal}{\mathcal{T}}

\newcommand{\prb}{\mathbb{P}}
\newcommand{\E}{\mathbb{E}}
\newcommand{\X}{\mathfrak{X}}
\newcommand{\Y}{\mathcal{Y}}
\newcommand{\Xcal}{\mathcal{X}}
\newcommand{\Zcal}{\mathcal{Z}}

\newcommand{\cov}{\mathrm{cov}}
\newcommand{\var}{\mathrm{var}}
\newcommand{\corr}{\mathrm{corr}}
\newcommand{\spn}{\mathrm{span}}

\newcommand{\ovl}{\overline}

\newcommand{\wt}{\widetilde}

\title{\textbf{On Tail Triviality of Negatively Dependent Stochastic Processes}}
\author{Kasra Alishahi
\thanks{Department of Mathematical Sciences, Sharif University of Technology. Email: \protect\url{alishahi@sharif.edu}.}
\and
Milad Barzegar
\thanks{Department of Mathematical Sciences, Sharif University of Technology. Email: \protect\url{milad.barzegar@sharif.edu}.}
\and
Mohammadsadegh Zamani
\thanks{Department of Statistics, Yazd University. Email: \protect\url{zamani@yazd.ac.ir}.}
}
\date{}

\begin{document}
\maketitle

\begin{abstract}
We prove that every negatively associated sequence of Bernoulli random variables with ``summable covariances'' has a trivial tail $\sigma$-field. A corollary of this result is the tail triviality of strongly Rayleigh processes. This is a generalization of a result due to Lyons which establishes tail triviality for discrete determinantal processes. We also study the tail behavior of negatively associated Gaussian and Gaussian threshold processes. We show that these processes are tail trivial though they do not in general satisfy the summable covariances property. Furthermore, we construct negatively associated Gaussian threshold vectors that are not strongly Rayleigh. This identifies a natural family of negatively associated measures that is not a subset of the class of strongly Rayleigh measures.
\end{abstract}

\section{Introduction}\label{Sec:Intro}

Kolmogorov's celebrated zero-one theorem states that every sequence of independent random variables has a trivial tail $\sigma$-field---that is, the probability of every tail event is either zero or one. A natural question in this context is: what other general classes of stochastic processes satisfy tail triviality? It is known that a stochastic process is tail trivial if and only if its distant ``parts'' are asymptotically independent  (see \cref{Prop:Georgii}). Since negatively dependent measures cannot have an everywhere-strong dependence structure (see \cite{Alishahi2020Paving} for a discussion), it is reasonable to expect asymptotic independence, hence tail triviality, in various classes of negatively dependent stochastic processes.

Tail triviality has been established for some well-known negatively dependent processes: Pemantle in \cite{Pemantle1991Choosing}, among other results, proved tail triviality for the free and the wired uniform spanning forests on Euclidean lattices. Benjamini et al. in their seminal paper \cite{Benjamini2001Uniform}, extended this result to general graphs. Later, Lyons \cite{Lyons2003Determinantal} showed that all discrete determinantal processes are tail trivial. This result has since been extended to general determinantal point processes (see, for instance, \cite{Lyons2018A}).

In this paper, we focus on $\{0,1\}$-valued ``negatively associated'' stochastic processes. Negative association, introduced by Joag-Dev and Proschan in \cite{Joag-Dev1983Negative}, is a natural notion of negative dependence. A stochastic process $\X = (X_i)_{i \in \N}$ is \textit{negatively associated} if $\E[fg] \leq \E[f] \, \E[g]$ for every pair of increasing functions $f$ and $g$ that depend on disjoint subsets of the components of $\X$. We will use the abbreviation NA to denote negative association. Many theorems known for sequences of independent random variables have been generalized to NA processes. See, for example, \cite{Newman1984Asymptotic, Shao1999Law, Shao2000Comparison, Zhang2001Weak, Matula1992Note, Barbour1992Poisson}.

We will identify a sufficient condition for the tail triviality of NA sequences of Bernoulli random variables. Interestingly, this property only involves the covariance structure. A stochastic process $\X = (X_i)_{i \in \N}$ has \textit{summable covariances} if $\sum_{j=1}^\infty \big|\cov(X_i,X_j)\big| < \infty$ for every $i \in \N$. We also recall that the \textit{tail $\sigma$-field} of $\X$, denoted $\Tcal(\X)$, is the $\sigma$-field $\bigcap_{n=1}^\infty \sigma(X_i : i \geq n)$. In other words, the tail $\sigma$-field is the family of all events whose occurrence is not affected by changing the values of any finite number of $X_i$'s.

\begin{thrm}\label{Th:TailTriviality}
	Let $\X = (X_i)_{i \in \N}$ be a sequence of negatively associated Bernoulli random variables. If $\X$ has summable covariances then its tail $\sigma$-field is trivial---namely, $\prb(E) \in \{0,1\}$ for every $E \in \Tcal(\X)$.
\end{thrm}

\cref{Sec:SummableCovs} is devoted to the proof of \cref{Th:TailTriviality} and its consequences. The most important corollary of this theorem is the tail triviality of ``strongly Rayleigh processes''. These processes satisfy $\sum_j \big|\cov(X_i,X_j)\big| \leq 1/2$ for every $i \in \N$ (see \cref{Subsec:SR}). Strongly Rayleigh measures, introduced by Borcea, Br\"{a}nd\'{e}n and Liggett in \cite{Borcea2009Negative}, are a very important class of negatively dependent measures. They exhibit strong negative dependence properties, including NA, and they include almost all of the known examples of negatively dependent measures, most notably determinantal processes. Thus our result generalizes Lyons' result on the tail triviality of determinantal processes. See \cite{Michelen2019Central, Gharan2011Randomized, Anari2014Kadison, Branden2012Negative, Pemantle2014Concentration, Alishahi2020Paving, Ghosh2017Multivariate} for more information on strongly Rayleigh measures and their applications.

In \cref{Subsec:Stationary}, we prove that NA weakly stationary processes satisfy the summable covariances property and hence are tail trivial. This is a generalization of a result due to Lebowitz \cite{Lebowitz1972Bounds} that establishes ergodicity for NA strongly stationary processes---also, see \cite{Newman1984Asymptotic}.

In \cref{Sec:Gaussian}, we study NA ``Gaussian threshold processes''. In the first part of this section, we deviate from $\{0,1\}$-valued processes and prove that NA Gaussian processes are tail trivial. Since thresholding preserves tail triviality, it follows that NA Gaussian threshold processes are tail trivial. In \cref{Subsec:GaussianThreshold}, we show that a Gaussian process $(Z_i)_{i \in \N}$ has summable covariances if and only if the threshold process $\big(1_{\{Z_i \geq 0\}}\big)_{i \in \N}$ has summable covariances. We use this result to construct an NA Gaussian threshold process that does not satisfy the summable covariances property. Hence the summable covariances property is not a necessary condition for tail triviality. Furthermore, we use this example to show that NA Gaussian threshold vectors are not in general strongly Rayleigh. This answers a question posed by Pemantle who, motivated by finding natural examples of NA measures that are not strongly Rayleigh, asked whether Gaussian threshold measures are included in the class of strongly Rayleigh measures.

A natural question regarding \cref{Th:TailTriviality} is whether tail triviality holds under weaker assumptions. For instance, does tail triviality hold for all NA $\{0,1\}$-valued processes? We currently do not know the answer to this question, but the following example shows that pairwise independence (and so pairwise negative correlations) is not sufficient even under strong stationarity: Let $\Xcal = (X_i)_{i \in \Z}$ be a pairwise independent (but not mutually independent) strongly stationary $\{0,1\}$-valued process such that $\prb(X_i=0) = \prb(X_i=1) = 1/2$ for all $i$ (see \cite{Robertson1985Pairwise} for a construction of such processes). Also, let $\Y = (Y_i)_{i \in \Z}$ be an i.i.d sequence of symmetric Bernoulli variables that is independent from $\Xcal$. Choose $p \in (0,1)$ and define $\Zcal$ to be equal to $\Xcal$ with probability $p$ and equal to $\Y$ with probability $1-p$. Note that $\Zcal$ is a pairwise independent stationary process such that $\prb(Z_i=0) = \prb(Z_i=1) = 1/2$ for all $i$. Since the law of $\Zcal$ is a strict convex combination of the laws of $\Xcal$ and $\Y$, it is not ergodic and hence not tail trivial.

\section{Tail triviality and summable covariances property}\label{Sec:SummableCovs}

In this section we prove \cref{Th:TailTriviality} and present two important corollaries: tail triviality for strongly Rayleigh processes and tail triviality for NA stationary $\{0,1\}$-valued processes.

\subsection{Proof of \texorpdfstring{\cref{Th:TailTriviality}}{Theorem 1.1}}\label{Subsec:ProofOfTailTriviality}

A $\{0,1\}$-valued stochastic process $(X_i)_{i \in S}$, where $S$ is finite or countably infinite, can equivalently be viewed as the random subset $\{i \in S : X_i = 1\}$ of $S$. Denoting this random subset by $\X$, we have $\X \cap A = (X_i)_{i \in A}$ and $|\X \cap A| = \sum_{i \in A} X_i$ for every $A \subseteq S$. From this point forth, we will alternatively use these representations without specifying which one is being used.

The idea of the proof of \cref{Th:TailTriviality} is as follows: A simple consequence of summable covariances is the asymptotic decorrelation of the number of ones in distant parts of the stochastic process (\cref{Lem:DecorrelationOfSize}). It turns out that, under negative association, this implies the asymptotic decorrelation of distant parts of the stochastic process (\cref{Prop:CovBound}). This property, in turn, is equivalent to tail triviality. The following proposition formalizes the last statement.

\begin{prop}[Proposition 7.9 of \cite{Georgii2011Gibbs}]\label{Prop:Georgii}
	A stochastic process $(X_i)_{i \in \N}$ is tail trivial if and only if for every cylinder event $E \in \sigma(X_i : i \in \N)$ we have
	\begin{align*}
		\lim_{N \to \infty} \sup_{F \in \sigma(X_i \, : \, i \geq N)} \big|\cov(1_E,1_F)\big| = 0.
	\end{align*}
\end{prop}

A more general version of the following result already appears in \cite{Bulinski1998Asymptotical}, but we include a proof for the sake of completeness.

\begin{prop}\label{Prop:CovBound}
	Let $X_1,\dots,X_n$ be NA Bernoulli random variables and $\X = (X_1,\dots,X_n)$. If $A,B \subseteq [n]$ are disjoint, then for any $E \in \sigma(\X \cap A)$ and $F \in \sigma(\X \cap B)$ we have
	\begin{align*}
		\big| \cov(1_E,1_F) \big| \leq \big| \cov(|\X \cap A|,|\X \cap B|) \big|.
	\end{align*}
\end{prop}

\begin{proof}
	Note that $|\X \cap A| \pm 1_E$ and $|\X \cap B| \pm 1_F$ are increasing functions of $\X$ with disjoint supports. Therefore, by negative association,
	\begin{align*}
		\cov\big(|\X \cap A| + 1_E, |\X \cap B| + 1_F\big) \leq 0
		\qquad\;
		\text{and}
		\qquad\;
		\cov\big(|\X \cap A| - 1_E, |\X \cap B| - 1_F\big) \leq 0.
	\end{align*}
	Adding these inequalities gives
	\begin{align*}
		\cov(1_E, 1_F) \leq - \cov\big(|\X \cap A|, |\X \cap B|\big).
	\end{align*}
	By replacing $1_F$ with $-1_F$ in the preceding argument, we arrive at
	\begin{align*}
		\cov(1_E, 1_F) \geq \cov\big(|\X \cap A|, |\X \cap B|\big).
	\end{align*}
	This completes the proof.
\end{proof}

\begin{lem}\label{Lem:DecorrelationOfSize}
	Let $\X = (X_i)_{i \in \N}$ be a sequence of Bernoulli random variables. If $\X$ has summable coviariances, then for every finite $A \subseteq \N$ and $\varepsilon \in \R_{>0}$, there exists $N \in \N$ such that
	\begin{align*}
		\big| \cov \big( |\X \cap A| , \big|\X \cap [N,\infty)\big| \big) \big| < \varepsilon.
	\end{align*}
\end{lem}

\begin{proof}
	Since $\X$ has summable covariances, for every $i \in A$ and $\varepsilon \in \R_{>0}$ there exists $N_i \in \N$ such that
	\begin{align*}
		\sum_{j=N_i}^\infty \big|\cov(X_i,X_j)\big| < \dfrac{\varepsilon}{|A|}.
	\end{align*}
	If $N = \max\{N_i : i \in A\}$ then
	\begin{align*}
		\big| \cov \big( |\X \cap A| , |\X \cap \{i : i \geq N\}| \big) \big|
		\leq \sum_{i \in A} \sum_{j=N}^\infty \big|\cov(X_i,X_j)\big|
		< \varepsilon.
	\end{align*}
\end{proof}

\begin{proof}[Proof of \cref{Th:TailTriviality}]
	Suppose that $E \in \sigma(\X)$ is a cylinder event. Hence there is a finite $A \subseteq \N$ such that $E \in \sigma(\X \cap A)$. Let $\varepsilon \in \R_{>0}$ and $N$ be given by \cref{Lem:DecorrelationOfSize}. Fix $F \in \sigma(X_i : i \geq N)$ and choose a large enough $M \in \N$ such that $M > N$ and there exists $\wt{F} \in \sigma\big(\X \cap [N,M]\big)$ that satisfies $\prb(F \triangle \wt{F}) < \varepsilon$. Note that $\big| \cov\big(|\X \cap A| , \big|\X \cap [N,M]\big|\big) \big| < \varepsilon$ and thus, by \cref{Prop:CovBound}, $\big|\cov(1_E,1_{\wt{F}})\big| < \varepsilon$. Therefore,
	\begin{align*}
		\big| \cov(1_E,1_F) \big| \leq \big| \cov(1_E,1_{\wt{F}}) \big| + \big| \cov(1_E,1_F-1_{\wt{F}}) \big| < \varepsilon + \big| \cov(1_E,1_F-1_{\wt{F}}) \big|.
	\end{align*}
	On the other hand,
	\begin{align*}
		\big| \cov(1_E,1_F-1_{\wt{F}}) \big| \leq 2 \, \E\big[ |1_F-1_{\wt{F}}| \big] \leq 2 \, \E\big[ 1_{F \triangle \wt{F}} \big] < 2\varepsilon.
	\end{align*}
	Hence $\big|\cov(1_E,1_F)\big| \leq 3\varepsilon$ and the result follows by \cref{Prop:Georgii}.
\end{proof}

\subsection{Strongly Rayleigh processes}\label{Subsec:SR}

Let $n$ be a positive integer. A probability measure $\mu$ on $\{0,1\}^n$ is \textit{strongly Rayleigh} if its generating polynomial, defined by
	\begin{align*}
		f_\mu(z_1,\dots,z_n) = \sum_{I \in \{0,1\}^n} \mu(I) \, \prod_{i=1}^{n} z_i^{I_i},
	\end{align*}
is ``real stable''---that is, it has no roots in $\mathbb{H}^n$, where $\mathbb{H}$ denotes the open upper half-plane. A vector $(X_1,\dots,X_n)$ of Bernoulli random variables is \textit{strongly Rayleigh} if its law is strongly Rayleigh. Furthermore, a sequence of Bernoulli variables $(X_i)_{i \in \N}$ is \textit{strongly Rayleigh} if $(X_1,\dots,X_n)$ is strongly Rayleigh for every finite $n \in \N$. See \cite{Borcea2009Negative} for more information.
	
	We will use the ``stochastic covering property'' of strongly Rayleigh measures to prove that strongly Rayleigh processes have summable covariances. For random subsets $\Y$ and $\Zcal$ of $[n]$, we say that $\Y$ \textit{stochastically covers} $\Zcal$ if there is a coupling $(\Y,\Zcal)$ that is supported on the set of pairs $(A,B)$ for which $A=B$ or $A \supseteq B$ with $|A \setminus B| = 1$. The stochastic covering property for strongly Rayleigh vectors is given by the following proposition.

\begin{prop}[Proposition 2.2 of \cite{Pemantle2014Concentration}]\label{Prop:SCP}
	Let $\X = (X_1,\dots,X_n)$ be strongly Rayleigh and $B \subseteq [n]$. If $U \subseteq V \subseteq B$ and $|V \setminus U| = 1$, then $[\X \, | \, \X \cap B = U]$ stochastically covers $[\X \, | \, \X \cap B = V]$.
\end{prop}

The following result is proved in \cite[Lemma 6.5]{Ghosh2017Multivariate}, but we recall its proof for the sake of completeness.

\begin{lem}\label{Lem:SRImpliesSummableCovs}
	If $(X_1,\dots,X_n)$ is strongly Rayleigh then
	\begin{align*}
		\forall i \in [n] \, : \, \sum_{j=1}^n \cov(X_i,X_j) \geq 0.
	\end{align*}
\end{lem}

\begin{proof}
	Let $i \in [n]$ and $S_i = \sum_{j \in [n] \setminus \{i\}} X_j$. We can assume $\prb(X_i=0) \, \prb(X_i=1) \neq 0$---otherwise the result is obvious. By \cref{Prop:SCP}, $\E[S_i \, | \, X_i=0] \leq \E[S_i \, | \, X_i=1] + 1$. It follows that $\E[X_i+S_i \, | \, X_i=0] \leq \E[X_i+S_i \, | \, X_i=1]$. Thus $\E[X_i+S_i \, | \, X_i]$ is an increasing function of $X_i$ and so $X_i+S_i$ and $X_i$ are positively correlated. This immediately implies the lemma.
\end{proof}

The preceding lemma easily extends to strongly Rayleigh processes---that is, if $(X_i)_{i \in \N}$ is strongly Rayleigh then $\sum_{j=1}^\infty \cov(X_i,X_j) \geq 0$ for all $i \in \N$. By pairwise negative correlations, it follows that
\begin{align}\label{Eq:SRCovSum}
	\forall i \in \N \, : \, \sum_{j=1}^\infty \big|\cov(X_i,X_j)\big| \leq 2\,\var(X_i) \leq \dfrac{1}{2}.
\end{align}
Thus strongly Rayleigh processes satisfy the summable covariances property.

\begin{cor}\label{Cor:SRTailTriviality}
	Strongly Rayleigh processes are tail trivial.
\end{cor}

\begin{rmrk}
	We only relied on the stochastic covering property of strongly Rayleigh measures in proof of \cref{Lem:SRImpliesSummableCovs}. Therefore, tail triviality holds for the slightly larger class of NA $\{0,1\}$-processes that satisfy the stochastic covering property.
\end{rmrk}

\subsection{Negatively associated Stationary processes}\label{Subsec:Stationary}

A stochastic process $(X_i)_{i \in S}$, where $S = \N$ or $\Z$, is \textit{weakly stationary} if for every $k \in S$ we have
\begin{align*}
	\forall i \in S \, : \, \E[X_i] &= \E[X_{i+k}],
	\\
	\forall i,j \in S \, : \, \cov(X_i,X_j) &= \cov(X_{i+k},X_{j+k}).
\end{align*}
Also, $(X_i)_{i \in S}$ is \textit{strongly stationary} if for every $k \in S$ and $n \in \N$ we have
\begin{align*}
	\forall i_1,\dots,i_n \in S \, : \, (X_{i_1+k},\dots,X_{i_n+k}) \sim (X_{i_1},\dots,X_{i_n}).
\end{align*}

We will use the following lemma to prove the summable covariances property for NA stationary processes.

\begin{lem}\label{Lem:PairwiseCovSumOfPNCMeasures}
	If $X_1,\dots,X_n$ are negatively correlated Bernoulli random variables, then
	\begin{align*}
		\sum_{i,j=1}^n \big|\cov(X_i,X_j)\big| \leq \dfrac{1}{2}n.
	\end{align*}
\end{lem}

\begin{proof}
	We have
	\begin{align*}
		0
		\leq \var \bigg( \dfrac{1}{\sqrt{n}} \sum_{i=1}^n X_i \bigg)
		= \dfrac{1}{n} \sum_{i=1}^n \var(X_i) + \dfrac{1}{n} \sum_{\substack{i,j=1 \\ i \neq j}}^n \cov(X_i,X_j).
	\end{align*}
	Therefore, by the pairwise negative correlations assumption, $\sum_{i \neq j} |\cov(X_i,X_j)| \leq \sum_i \var(X_i)$. The lemma follows because $\var(X_i) \leq 1/4$ for all $i \in [n]$.
\end{proof}

A corollary of this lemma is that if $(X_i)_{i \in S}$, where $S = \N$ or $\Z$, is weakly stationary and negatively correlated, then
\begin{align*}
	\forall i \in S \, : \, \sum_{j \in S} \big|\cov(X_i,X_j)\big| \leq \dfrac{1}{2}.
\end{align*}
To see this, note that
\begin{align*}
	\var(X_i) + 2\sum_{\substack{j=1 \\ j \neq i}}^\infty \big|\cov(X_i,X_j)\big| = \lim_{n \to \infty} \dfrac{1}{n} \sum_{j,k=1}^n \big|\cov(X_j,X_k)\big| \leq \dfrac{1}{2}.
\end{align*}
Thus negatively correlated weakly stationary processes satisfy the summable covariances property.

\begin{cor}\label{Cor:StationaryTailTriviality}
	NA weakly stationary $\{0,1\}$-valued processes are tail trivial.
\end{cor}

Since tail triviality implies ergodicity, this corollary is a generalization of a result due to Lebowitz \cite{Lebowitz1972Bounds} which establishes ergodicity for NA (strongly) stationary $\{0,1\}$-valued processes.

\begin{rmrk}
	With appropriate modifications to \cref{Prop:CovBound}, it is possible to extend \cref{Th:TailTriviality} to bounded NA processes. \cref{Lem:PairwiseCovSumOfPNCMeasures} also easily extends to bounded negatively correlated processes. Thus \cref{Cor:StationaryTailTriviality}, in fact, holds true for bounded NA weakly stationary processes.
\end{rmrk}

\section{Tail triviality of Gaussian and Gaussian threshold processes}\label{Sec:Gaussian}

In this section, we focus on Gaussian threshold processes. In \cref{Subsec:GaussianTailTriviality}, we prove that NA Gaussian processes are tail trivial. This implies tail triviality for NA Gaussian threshold processes. In \cref{Subsec:GaussianThreshold}, we construct an NA Gaussian threshold process that does not satisfy the summable covariances property. We also show that NA Gaussian threshold vectors are not in general strongly Rayleigh. This identifies a natural class of NA measures that are not strongly Rayleigh.

\subsection{Negatively associated Gaussian processes}\label{Subsec:GaussianTailTriviality}

It is well-known that for Gaussian processes, all notions of negative dependence are equivalent (see \cite{Joag-Dev1983Negative}). In particular, a Gaussian process is NA if and only if it is negatively correlated. In this subsection, we prove the following theorem.

\begin{thrm}\label{Th:GaussianTailTriviality}
	If $\Zcal = (Z_i)_{i \in \N}$ is a Gaussian process and $\cov(Z_i,Z_j) \leq 0$ for all distinct $i,j \in \N$, then $\Zcal$ has a trivial tail $\sigma$-field.
\end{thrm}

The idea of the proof is as follows: Using a result due to Kolmogorov (\cref{Th:Hotelling}), we show that the decorrelation of distant parts of a Gaussian process is equivalent to the decorrelation of linear functions on its distant parts. This reduces the question of tail triviality to a ``linear algebraic tail triviality'', which we prove in \cref{Prop:LinearAlgebraicTailTriviality} and \cref{Cor:LinearAlgebraicDecorrelation}.

\begin{thrm}\label{Th:Hotelling}
	Let $X_1,\dots,X_m,Y_1,\dots,Y_n$ be jointly Gaussian and define $\X = (X_1,\dots,X_m)$ and $\Y = (Y_1,\dots,Y_n)$. Then, the maximal correlation coefficient between $\X$ and $\Y$, defined by
	\begin{align*}
		\rho_{max}(\X,\Y) = \sup_{\substack{\varphi \in L^2(\X) \\ \psi \in L^2(\Y)}} \corr \big( \varphi(\X) , \psi(\Y) \big),
	\end{align*}
	is attained on linear functions.
\end{thrm}

see \cite{Lancaster1966Kolmogorov} for a proof of the preceding theorem. The next proposition states that a sequence of vectors with pairwise nonpositive inner products has a trivial ``tail linear span''.

\begin{prop}\label{Prop:LinearAlgebraicTailTriviality}
	Let $(v_i)_{i \in \N}$ be a sequence of vectors in an inner product space. If $\langle v_i, v_j \rangle \leq 0$ for all distinct $i, j \in \N$, then
	\begin{align*}
		\bigcap_{n=1}^\infty \, \ovl{\spn}\{v_i : i \geq n\} = \{0\},
	\end{align*}
	where $\ovl{\spn}\{v_i : i \geq n\}$ denotes the closed linear span of $\{v_i : i \geq n\}$.
\end{prop}

\begin{proof}
	Let $u \in \bigcap_n \ovl{\spn}\{v_i : i \geq n\}$ and $\varepsilon \in \R_{>0}$. Thus for every $n \in \N$, there exists $u_n \in \spn\{v_i : i \geq n\}$ such that $\|u-u_n\| < \varepsilon$. Let $C_n$ denote the cone of all (finite) nonnegative linear combinations of $\{v_i : i \geq n\}$. Note that we can write $u_n = u_n^+-u_n^-$, where $u_n^+, u_n^- \in C_n$ and they can be expressed as conic combinations of disjoint subsets of $\{v_i : i \geq n\}$. In particular, $\langle u_n^+, u_n^- \rangle \leq 0$. Let $i_n$	 denote the largest index $i$ such that $v_i$ appears in at least one of the aforementioned conic combinations for $u_n^+$ and $u_n^-$.
		
	Suppose $m > i_1$. Note that $\|u_m-u_1\| < 2\varepsilon$ and thus $\|(u_m^++u_1^-) - (u_m^-+u_1^+)\| < 2\varepsilon$. On the other hand, since $m > i_1$, the vectors $u_m^++u_1^-$ and $u_m^-+u_1^+$ can be expressed as conic combinations of disjoint subsets of $\{v_i : i \in \N\}$. Therefore $\langle u_m^++u_1^-, u_m^-+u_1^+ \rangle \leq 0$, which implies
	\begin{align}
		\|(u_m^++u_1^-) - (u_m^-+u_1^+)\|^2
		&= \|u_m^++u_1^-\|^2 + \|u_m^-+u_1^+\|^2 - 2 \, \langle u_m^++u_1^-, u_m^-+u_1^+ \rangle \nonumber
		\\
		&\geq \|u_m^++u_1^-\|^2 + \|u_m^-+u_1^+\|^2. \label{Eq:ProofOfLinearAlgebraicTailTriviality}
	\end{align}
	In particular, $\|u_m^++u_1^-\| < 2\varepsilon$ and $\|u_m^-+u_1^+\| < 2\varepsilon$. Now suppose $k > i_m$. Since $i_m > i_1$, the same inequalities hold with $m$ replaced by $k$. It follows that $\|u_m^+ - u_k^+\| < 4\varepsilon$ and $\|u_m^- - u_k^-\| < 4\varepsilon$. Since $k > i_m$, we have $\langle u_m^+ , u_k^+ \rangle \leq 0$ and $\langle u_m^- , u_k^- \rangle \leq 0$. By an argument similar to \eqref{Eq:ProofOfLinearAlgebraicTailTriviality}, we can deduce that $\|u_m^+\| < 4\varepsilon$ and $\|u_m^-\| < 4\varepsilon$, which implies $\|u\| < 9\varepsilon$. This inequality holds for every $\varepsilon \in \R_{>0}$ and thus we have $u = 0$.
\end{proof}

\begin{cor}\label{Cor:LinearAlgebraicDecorrelation}
	Let $(v_i)_{i \in \N}$ be a sequence of vectors in an inner product space. If $\langle v_i,v_j \rangle \leq 0$ for all distinct $i,j \in \N$, then for every $n \in \N$ and $\varepsilon \in \R_{>0}$, there exists $N \in \N$ such that
	\begin{align}\label{Eq:LinearAlgebraicDecorrelation}
		\forall u \in \spn\{v_1,\dots,v_n\} \, , \, \forall w \in \ovl{\spn}\{v_i : i \geq N\} \, : \, \big|\langle u,w \rangle\big| < \varepsilon \|u\| \|w\|.
	\end{align}
\end{cor}

\begin{proof}
	For every $m \in \N$, let $W_m = \ovl{\spn}\{v_i : i \geq m\}$ and denote the orthogonal projection onto $W_m$ by $P_{m}$. Note that \eqref{Eq:LinearAlgebraicDecorrelation} holds if and only if $\|P_N u\| < \varepsilon$ for every $u \in \spn\{v_1,\dots,v_n\}$ such that $\|u\|=1$. Suppose to the contrary that for every $m \in \N$, there exists $u_m \in \spn\{v_1,\dots,v_n\}$ such that $\|u_m\| = 1$ and $\|P_m u_m\| \geq \varepsilon$. Since the unit ball in $\spn\{v_1,\dots,v_n\}$ is compact, $(u_m)_{m \in \N}$ has a convergent subsequent, which we denote by $(u_{m_i})_{i \in \N}$. Suppose $u_{m_i} \to u$. Note that $\|P_k u_m\| \geq \|P_m u_m\|$ for any $k$ and $m$ such that $m \geq k$. It follows that $\|P_k u\| = \lim_{i \to \infty} \|P_k u_{m_i}\| \geq \varepsilon$ for all $k \in \N$. But this is in contradiction with \cref{Prop:LinearAlgebraicTailTriviality} because $\lim_{k \to \infty} P_k$ is the orthogonal projection onto $\bigcap_{k=1}^\infty W_k$.
\end{proof}

Now, we are ready to complete the proof of \cref{Th:GaussianTailTriviality}.

\begin{proof}[Proof of \cref{Th:GaussianTailTriviality}]
	Suppose $E \in \sigma(\Zcal)$ is a cylinder event. Thus there exists $n \in \N$ such that $E \in \sigma(Z_1,\dots,Z_n)$. Let $\varepsilon \in \R_{>0}$. By treating $Z_i$'s as the elements of a Hilbert space with inner product $\cov(\cdot,\cdot)$, \cref{Cor:LinearAlgebraicDecorrelation} implies that there exists $N \in \N$ such that
	\begin{align*}
		\sup_{\substack{X \in \spn(Z_1,\dots,Z_n) \\ Y \in \spn(Z_i \, : \, i \geq N)}} \big|\corr(X ,Y)\big| < \varepsilon.
	\end{align*}
	
	Now, suppose $F \in \sigma(Z_i : i \geq N)$. Choose $M \in \N$ large enough such that $M>N$ and there exists $\wt{F} \in \sigma(Z_N,\dots,Z_M)$ that satisfies $\prb(F \triangle \wt{F}) < \varepsilon$. By \cref{Th:Hotelling}, $|\corr(1_E,1_{\wt{F}})| < \varepsilon$. Since $|\corr(1_E,1_{\wt{F}})| \geq 4 \, |\cov(1_E,1_{\wt{F}})|$, we have $|\cov(1_E,1_{\wt{F}})| <\varepsilon/4$. Therefore,
	\begin{align*}
		\big| \cov(1_E,1_F) \big| \leq \big| \cov(1_E,1_{\wt{F}}) \big| + \big| \cov(1_E,1_F-1_{\wt{F}}) \big| < \dfrac{\varepsilon}{4} + \big| \cov(1_E,1_F-1_{\wt{F}}) \big|.
	\end{align*}
	On the other hand,
	\begin{align*}
		\big| \cov(1_E,1_F-1_{\wt{F}}) \big| \leq 2 \, \E\big[ |1_F-1_{\wt{F}}| \big] \leq 2 \, \E\big[ 1_{F \triangle \wt{F}} \big] < 2\varepsilon.
	\end{align*}
	Hence $\big|\cov(1_E,1_F)\big| \leq 9\varepsilon/4$ and the theorem follows by \cref{Prop:Georgii}.
\end{proof}

\subsection{Negatively associated Gaussian threshold processes}\label{Subsec:GaussianThreshold}

An immediate consequence of the tail triviality of NA Gaussian processes is the tail triviality of NA Gaussian threshold processes. In this subsection, we construct an NA Gaussian threshold process that is tail trivial but does not have summable covariances. Thus, in the case of NA $\{0,1\}$-valued processes, the summable covariances property is not a necessary condition for tail triviality. Furthermore, this example shows that the class of NA Gaussian threshold processes is not a subset of the class of strongly Rayleigh measures.

A stochastic process $(X_i)_{i \in \N}$ is a \textit{Gaussian threshold process} if there exists a Gaussian process $(Z_i)_{i \in \N}$ and a sequence of real numbers $(a_i)_{i \in \N}$ such that $X_i = 1_{\{Z_i \geq a_i\}}$ for all $i \in \N$. The following proposition shows that a Gaussian threshold process is NA if and only if its parent Gaussian process is NA.

\begin{prop}\label{Prop:NAGaussianThreshold}
	Let $(Z_i)_{i \in \N}$ be a Gaussian process, $(a_i)_{i \in \N}$ be a sequence of real numbers and $X_i = 1_{\{Z_i \geq a_i\}}$ for all $i \in \N$. The stochastic process $(X_i)_{i \in \N}$ is NA if and only $(Z_i)_{i \in \N}$ is NA.
\end{prop}

\begin{proof}
	Since every coordinatewise increasing function of $(X_i)_{i \in \N}$ lifts to a coordinatewise increasing function of $(Z_i)_{i \in \N}$, negative association of $(Z_i)_{i \in \N}$ implies negative association of $(X_i)_{i \in \N}$. Now we consider the ``only if'' direction. Since $X_k$ is an increasing function of $Z_k$ for all $k \in \N$, we have $\cov(X_i,X_j) \leq 0$ if and only if $\cov(Z_i,Z_j) \leq 0$. Therefore, if $(X_i)_{i \in \N}$ are NA (and thus negatively correlated) then $(Z_i)_{i \in \N}$ are negatively correlated. Hence, by the theorem of Joag-Dev and Proschan, $(Z_i)_{i \in \N}$ is NA.
\end{proof}

The following lemma estimates the effect of thresholding at 0 on the covariance structure of a Gaussian process.

\begin{lem}\label{Lem:GaussianThresholdCov}
	Let $(Y_1,Y_2) \in \R^2$ be a Gaussian vector with $\E[Y_1] = \E[Y_2] = 0$ and $\var(Y_1) = \var(Y_2) = 1$. If $\cov(Y_1,Y_2) = \rho$ then $\cov\big(1_{\{Y_1\geq0\}},1_{\{Y_2\geq0\}}\big) = (1/2\pi) \rho + o(\rho)$ as $\rho \to 0$.
\end{lem}

\begin{proof}
	Note that
	\begin{align*}
		(Y_1,Y_2) \sim \big(Z_1 , \rho Z_1+\sqrt{1-\rho^2} Z_2\big),
	\end{align*}
	where $(Z_1,Z_2) \in \R^2$ is a standard Gaussian vector. Therefore,
	\begin{align}\label{Eq:GaussianThresholdCov}
		\cov\big(1_{\{Y_1 \geq 0\}},1_{\{Y_2 \geq 0\}}\big)
		&= \prb(Y_1 \geq 0, Y_2 \geq 0) - \prb(Y_1 \geq 0) \, \prb(Y_2 \geq 0) \nonumber
		\\
		&= \prb\bigg( Z_1 \geq 0, \dfrac{Z_2}{Z_1} \geq \dfrac{-\rho}{\sqrt{1-\rho^2}} \bigg) - \dfrac{1}{4}.
	\end{align}
	Since $(Z_1,Z_2)$ is symmetrically distributed with respect to the origin,
	\begin{align*}
		\prb\bigg( Z_1 \geq 0, \dfrac{Z_2}{Z_1} \geq \dfrac{-\rho}{\sqrt{1-\rho^2}} \bigg) = \dfrac{1}{2} \; \prb\bigg( \dfrac{Z_2}{Z_1} \geq \dfrac{-\rho}{\sqrt{1-\rho^2}} \bigg).
	\end{align*}
	Substituting in \eqref{Eq:GaussianThresholdCov} gives
	\begin{align*}
		\cov\big(1_{\{Y_1 \geq 0\}},1_{\{Y_2 \geq 0\}}\big)
		&= \dfrac{1}{2} \Bigg( 1 - \bigg( \dfrac{1}{\pi} \arctan\!\bigg(\dfrac{-\rho}{\sqrt{1-\rho^2}}\bigg) + 1 \bigg) \Bigg) - \dfrac{1}{4}
		\\
		&= \dfrac{1}{2\pi} \, \arcsin(\rho).
	\end{align*}
	The lemma follows since $\arcsin(\rho) = \rho + o(\rho)$ as $\rho \to 0$.
\end{proof}

An immediate consequence of this lemma is that an NA Gaussian process $(Z_i)_{i \in \N}$ has summable covariances if and only if $\big( 1_{\{Z_i \geq 0\}} \big)_{i \in \N}$ has summable covariances. In the following example, we use this fact to construct an NA Gaussian threshold process that does not satisfy the summable covariances property.

\begin{ex}
	Suppose that $(Z_i)_{i \in \N}$ is a standard Gaussian process. For every $i \in \N$, define
\begin{align*}
	Y_i = -\dfrac{1}{i}Z_1 - \dots - \dfrac{1}{i}Z_{i-1} + Z_i.
\end{align*}
	Note that $(Y_i)_{i \in \N}$ is a Gaussian process and $\cov(Y_i,Y_j) = -1/ij$ for all distinct $i,j \in \N$. Therefore, $(Y_i)_{i \in \N}$ is NA and thus, by \cref{Th:GaussianTailTriviality}, it is tail trivial. Also, we have $\sum_j \big|\cov(Y_i,Y_j)\big| = \infty$ for all $i \in \N$.

	Define $X_i = 1_{\{Y_i \geq 0\}}$ for every $i \in \N$. By \cref{Prop:NAGaussianThreshold}, $(X_i)_{i \in \N}$ is NA. Also, it inherits tail triviality from $(Y_i)_{i \in \N}$. By \cref{Lem:GaussianThresholdCov}, $\sum_j \big|\cov(X_i,X_j)\big| = \infty$ for all $i \in \N$. Thus $(X_i)_{i \in \N}$ is NA and tail trivial, but it does not satisfy the summable covariances property.
\end{ex}

For the process $(X_i)_{i \in \N}$, constructed above, we have $\sum_{j=1}^n \big|\cov(X_i,X_j)\big| = \Theta(\log(n))$ for all $i \in \N$. A natural question is: How large can the growth rate of the covariance sums of an NA process be? The following lemma provides a bound.

\begin{lem}\label{Lem:CovSumOfPNCMeasures}
	If $X_1,\dots,X_n$ are negatively correlated Bernoulli random variables, then
	\begin{align*}
		\forall i,n \in \N \, : \, \sum_{j=1}^n \big| \cov(X_i,X_j) \big| \leq \dfrac{1}{4}+\dfrac{3}{4}\sqrt{n}.
	\end{align*}
\end{lem}

\begin{proof}
	We have
	\begin{align*}
		0
		\leq \var \Big( X_i + \dfrac{1}{\sqrt{n}} \sum_{\substack{j=1 \\ j \neq i}}^n X_j \Big)
		= \var(X_i) + \dfrac{1}{n} \sum_{\substack{j=1 \\ j \neq i}}^n \var(X_j) &+ \dfrac{1}{\sqrt{n}} \sum_{\substack{j=1 \\ j \neq i}}^n \cov(X_i,X_j)
		\\
		&+ \dfrac{1}{n} \sum_{\substack{j,k = 1 \\ j,k \neq i}}^n \cov(X_k,X_j).
	\end{align*}
	The lemma follows using \cref{Lem:PairwiseCovSumOfPNCMeasures} and the fact that $\var(X_i) \leq 1/4$ for all $i \in [n]$.
\end{proof}

The preceding lemma implies that the growth rate of the covariance sums of a negatively correlated processes can be at most of order $\sqrt{n}$. In the next example, we construct a vector of NA Bernoulli random variables that attains this bound.

\begin{ex}
	Let $(Z_1,\dots,Z_n)$ be a standard Gaussian vector and define $Z_0 = -(1/\sqrt{n}) \sum_i Z_i$. We have $\cov(Z_0,Z_i) = -(1/\sqrt{n})$ for every $i \in [n]$ and $\cov(Z_i,Z_j) = 0$ for all distinct $i,j \in [n]$. Therefore, $(Z_0,\dots,Z_n)$ is NA and $\sum_{i=1}^n \cov(Z_0,Z_i) = -\sqrt{n}$. Now, define $X_i = 1_{\{Z_i \geq 0\}}$ for $i = 0,\dots,n$. By \cref{Prop:NAGaussianThreshold}, $(X_0,\dots,X_n)$ is NA and by \cref{Lem:GaussianThresholdCov}, $\sum_{i=1}^n \big|\cov(X_0,X_i)\big| = \Theta\big(\sqrt{n}\big)$.
\end{ex}

By \eqref{Eq:SRCovSum}, the Gaussian threshold vector $(X_0,\dots,X_n)$, constructed in the preceding example, is not strongly Rayleigh for large $n$---in fact, this is true for $n \geq 3$.

\begin{cor}\label{Cor:NANotSR}
	NA Gaussian threshold measures are not in general strongly Rayleigh.
\end{cor}

It is known that the strong Rayleigh property is strictly stronger than negative association (see \cite{Borcea2009Negative}); however, the currently known examples are highly contrived. Pemantle raised the question of finding natural classes of NA measures that are not strongly Rayleigh. He proposed the class of NA Gaussian threshold measures as a possible candidate. \cref{Cor:NANotSR} shows that these measures indeed are not a subset of strongly Rayleigh measures.

\subsubsection*{Acknowledgment}

We would like to thank Jeffrey Steif for suggesting the example at the end of the introduction, which shows that tail triviality does not follow from pairwise independence even under strong stationarity. We also wish to thank Yuzhou Gu who pointed out a mistake in an earlier version of this paper.

\printbibliography

\end{document}